\documentclass{amsart}

\usepackage[utf8]{inputenc} 
\usepackage{amsmath, amsthm, amssymb, booktabs, multirow, graphicx, booktabs, mathtools}
\usepackage{dialogue}
\usepackage[stable]{footmisc}
\usepackage{tikz,nicefrac}
\usetikzlibrary{calc,matrix,arrows,decorations.markings}
\usepackage[colorlinks]{hyperref}
\usetikzlibrary{
  graphs,
  graphs.standard
}

\newtheorem{defin}{Definition}[section]

\newtheorem{theorem}[defin]{Theorem}

\newtheorem{lemma}[defin]{Lemma}

\newtheorem{definition}[defin]{Definition}

\newcommand{\R}{\mathbb{R}}

\newcommand{\Ccal}{\mathcal{C}}

\DeclareMathOperator{\ILP}{ILP}
\DeclareMathOperator{\las}{las}
\DeclareMathOperator{\cone}{cone}
\DeclareMathOperator{\sym}{sym}

\begin{document}

\title{A semidefinite programming hierarchy for covering problems in
  discrete geometry}

\author{Cordian Riener}
\author{Jan Rolfes}
\author{Frank Vallentin}

\address{
  Cordian Riener,
  Department of Mathematics and Statistics, UiT The Arctic University
  of Norway,
  9037 Troms\o,
  Norway
}
\email{cordian.riener@uit.no}

\address{
  Jan Rolfes,
  Department of Mathematics, 
  Linköping University,
  SE-581 83 Linköping,
  Sweden}
\email{jan.rolfes@liu.se}

\address{
  Frank Vallentin,
  Universit\"at zu K\"oln, Department Mathematik/Informatik, Abteilung
  Mathematik, Weyertal 86-90, 50931 K\"oln,  Germany}
\email{frank.vallentin@uni-koeln.de}

\date{May 2, 2025}

\begin{abstract}
  In this paper we present a new semidefinite programming hierarchy
  for covering problems in compact metric spaces.

  Over the last years, these kind of hierarchies were developed
  primarily for geometric packing and for energy minimization
  problems; they frequently provide the best known bounds.

  Starting from a semidefinite programming hierarchy for the
  dominating set problem in graph theory, we derive the new hierarchy
  for covering and show some of its basic properties: The hierarchy
  converges in finitely many steps, but the first level collapses to
  the volume bound when the compact metric space is homogeneous.
\end{abstract}

\maketitle

\markboth{C.~Riener, J.~Rolfes, F.~Vallentin}{A semidefinite
  programming hierarchy for covering problems in discrete geometry}

\section{Introduction}
\label{sec:introduction}

Let $(X,d)$ be a compact metric space. Frequently one asks for specific
configurations of finitely many points in $X$. Various performance
measures for such finite point distributions are
investigated. Homogeneous measures, like maximizing minimal distance
between pairs of distinct points (\textit{packing}), or minimizing
energy with respect to a potential function (\textit{energy
  minimization}), are most commonly used. These homogeneous measures
have inhomogeneous counterparts: Finding a point configuration which
minimizes the maximal distance of an arbitrary point in $X$ to the
finite point distribution (\textit{covering}), or maximizing the
minimal potential energy of an arbitrary point in $X$ in the force
field given by the charges of the point configuration (\textit{max-min
  polarization}). For a more detailed overview, we refer to the book
of Borodachov, Hardin, Saff \cite{Borodachov-Hardin-Saff-2019}.

\medskip

Over the last years, semidefinite programming hierarchies, in the
spirit of the moment-SOS hierarchy in polynomial optimization, due to
Lasserre \cite{Lasserre-2001} and Parrilo \cite{Parrilo-2003}, were developed primarily for providing bounds for geometric
packing (de Laat, Vallentin
\cite{de-Laat-Vallentin-2015}) and for energy minimization problems (de Laat
\cite{de-Laat-2020}). Computing these semidefinite programming bounds
frequently yield the best known bounds.

\medskip

In this paper we want to demonstrate that this approach is also
applicable to geometric covering problems.

\medskip

We start by defining the covering problem formally. Let $X$ be a
compact metric space with distance function $d$. We denote closed
metric balls of radius $r$ and center $x$ by
$\overline{B}(x,r) = \{x' \in X : d(x,x') \leq r\}$.  A set of balls of
radius $r$ is uniquely determined by the collection of its centers
$Z$.  A set $Z$ of centers with the property
$ \bigcup_{z \in Z} \overline{B}(z,r) = X$ is called a \emph{covering} of $X$.
Since $X$ is compact, there always exists a finite covering. The
\emph{covering number} $\mathcal{N}(X,r)$ of the space $X$ and a
positive number $r$ is the smallest number of such balls with radius
$r$ one needs to cover~$X$. So,
\[
  \mathcal{N}(X,r) = \min\left\{ |Z| : Z \subseteq X, \bigcup_{z \in
      Z} \overline{B}(z,r) = X\right\}.
\]

Determining the covering number is a fundamental problem in metric
geometry (see for example the classical book by Rogers
\cite{Rogers-1964}). Among others, applications arise in the field of
compressive sensing \cite{Foucart-Rauhut-2013}, approximation theory and
machine learning \cite{Cucker-Smale-2002}, in probability theory
\cite{Ledoux-Talagrand-1991} and theoretical quantum computing
\cite{Nielsen-Chuang-2000}.

So far, upper bounds for the covering number of several specific
metric spaces are known. For a survey on this, we refer to Nasz\'odi
\cite{Naszodi-2018}. Using a greedy approach, in the spirit
Chv\'atal's greedy approximation algorithm for the set-covering
problem, Rolfes and Vallentin \cite{Rolfes-Vallentin-2018} provide
upper bounds for a wider class of compact metric spaces.  They showed
that for every $\varepsilon$ with $r/2 > \varepsilon > 0$ the inequalities 
\begin{equation}\label{Eq: ResultGreedypaper}
	\frac{1}{\omega_r} \leq \mathcal{N}(X,r) \leq
	\frac{1}{\omega_{r-\varepsilon}}
	\left(\ln\left(\frac{\omega_{r-\varepsilon}}{\omega_\varepsilon}\right)+1\right),
      \end{equation}
      holds, where $X$ is equipped with a probability measure $\omega$
      satisfying the following two conditions:
\begin{enumerate}
	\item[(a)] $\omega(\overline{B}(x,s))=\omega(\overline{B}(x',s))$ for all $x,x'\in X$, and for
	all $s \geq 0$,
	\item[(b)] $\omega(\overline{B}(x,\varepsilon))>0$ for all $x\in X$, and for all $\varepsilon>0$.
\end{enumerate}
By (a) the measure of a ball does only depend on the radius $s$ and
not on the center $x$, so we simply denote $\omega(\overline{B}(x,s))$
by $\omega_s$. These conditions are fulfilled when $X$ is a
\emph{homogeneous} space, that is, when the automorphism group acts
transitively. Then the Haar measure of the automorphism group induces
an invariant probability measure on $X$ satisfying (a) and (b).

The trivial lower bound above is known as the \emph{volume
  bound}. Stronger and more sophisticated lower bounds have not yet
been established.  The central aim of this paper is to develop a
semidefinite programming hierarchy for geometric covering problems
giving such lower bounds.

\medskip

The paper is structured as follows: 

We begin by motivating our definition of the covering hierarchy by
recalling the Lasserre hierarchy for general $0/1$ integer linear
programs in Section~\ref{sec:Lasserre-hierarchy}. We
hope that our detailed presentation, including complete proofs, of this well-known material will make our 
somewhat technical new hierarchy accessible to non-specialists.

Then we also recall
the packing hierarchy of de Laat, Vallentin
\cite{de-Laat-Vallentin-2015} in
Section~\ref{sec:packing-hierarchy}. The covering hierarchy is based
on the packing hierarchy.

We will define our covering hierarchy in
Section~\ref{sec:covering-hierarchy} and show its basic properties in
Theorem~\ref{thm:main}, which is the principal result of the paper.

In Section~\ref{sec:dual-covering-hierarchy} we derive the dual covering
hierarchy and use the dual to show that for homogeneous compact metric
spaces the first step of the hierarchy collapses to the volume bound.

\section{The Lasserre hierarchy for $0/1$ integer linear programs}
\label{sec:Lasserre-hierarchy}

Consider the following integer linear program with $0/1$-variables
\[
\ILP =  \min\{c^{\sf T} x : x \in \{0,1\}^n, \; Ax \geq b\},
\]
where the matrix $A \in \R^{m \times n}$ and the vectors $b \in \R^m$
and $c \in \R^n$ are given as input.

In the by now classical paper \cite{Lasserre-2002} Lasserre showed how
to systematically compute a sequence of stronger and stronger lower
bounds, converging in $n+1$ steps to $\ILP$, by solving semidefinite
programs of growing size.

Lasserre's original convergence proof relies on the Positivstellensatz
of Putinar~\cite{Putinar-1993} and on the flat extension theorem of
Curto, Fialkow~\cite{Curto-Fialkow-2000}. Laurent~\cite{Laurent-2003}
provided an elementary, combinatorial proof of the convergence. We
shall reproduce her arguments below because they gave the inspiration
for the semidefinite hierarchies for geometric packing and covering
problems. For good behavior we note that the results of Lasserre and
Laurent are much more general, as they apply to nonlinear (polynomial)
optimization problems with $0/1$-variables.

\subsection{Combinatorial moment and localizing matrices}

To define the Lasserre hierarchy for solving $\ILP$ we need a bit of
notation.

Set $X = [n] = \{1, \ldots, n\}$. We write $\mathcal{P}(X)$ for the
power set of $X$.  By $\mathcal{P}_t(X)$ we denote the subsets of
$\mathcal{P}(X)$ containing all $J \in \mathcal{P}(X)$ with
$|J| \leq t$. For $t = 0, \ldots, n$ and $J \subseteq X$ we define the
\emph{(truncated) characteristic vector}
$\chi^J_t \in \{0,1\}^{\mathcal{P}_t(X)}$ by
\[
  \chi^J_t(J') =
  \begin{cases}
    1 & \text{if $J' \subseteq J$,}\\
    0 & \text{otherwise.}
    \end{cases}
\]

A vector $y \in \R^{\mathcal{P}_{2t}(X)}$ defines the
\emph{(truncated) combinatorial moment matrix}
$M_t(y) \in \mathcal{S}^{\mathcal{P}_t(X)}$, which is a symmetric
matrix with $|\mathcal{P}_t(X)| = \sum_{i=0}^t \binom{n}{i}$ many rows and columns, by
\[
[M_t(y)]_{J,J'} = y(J \cup J'), \quad \text{ with } J, J' \in \mathcal{P}_t(X).
\]

The following lemma dealing with $t = n$, the final step of Lasserre's
hierarchy, is crucial. It is due to Lindstr\"om \cite{Lindstroem-1969}
and Wilf \cite{Wilf-1968}.

\begin{lemma}
  \label{lem:crucial-lemma}
  A vector $y \in \R^{\mathcal{P}(X)}$ defines a positive semidefinite
  combinatorial moment matrix $M_n(y)$ if and only if $y$ lies in the
  (simplicial) polyhedral cone generated by the characteristic vectors
  $\chi^J_n$, with $J \in \mathcal{P}(X)$. In short:
  \[
    M_n(y) \succeq 0 \Longleftrightarrow y \in \cone\{\chi^J_n : J \in
    \mathcal{P}(X)\}.
  \]
\end{lemma}

\begin{proof}
Sufficiency follows easily by the \emph{homomorphism property} of the
characteristic vectors, which is
\[
  \chi^J_n(J' \cup J'') =   \chi^J_n(J')  \cdot  \chi^J_n( J'') \; \text{
    for all } \; J, J', J'' \in \mathcal{P}(X).
\]
Now any $y$ with
\[
  y  = \sum_{J \in \mathcal{P}(X)} \alpha_J \chi^J_n \; \text{ and } \;
  \alpha_J \geq 0
\]
defines a positive semidefinite combinatorial moment matrix
\[
M_n(y) = \sum_{J \in \mathcal{P}(X)} \alpha_J \chi^J_n (\chi^J_n)^{\sf T}.
\]

For necessity, we first note that the characteristic vectors
$\chi^J_n$ form a basis of $\mathbb{R}^{\mathcal{P}(X)}$. Let
$\psi_n^J$ be the dual basis\footnote{In the following we only need the defining property of the dual basis. The dual basis can also explicity given by the formula $\psi^{J}_n(J') =
\begin{cases}
  (-1)^{|J' \setminus J|} & \text{if $J \subseteq J'$,}\\
  0 & \text{otherwise.}
\end{cases}
$} with respect to the standard inner
product, so that
\[
  (\chi^J_n)^{\sf T} \psi^{J'}_n = \delta_{J,J'}, \quad \text{where}
  \quad 
  \delta_{J,J'} = 
  \begin{cases} 
    1 & \text{if $J=J'$,} \\ 
    0 & \text{otherwise.}
\end{cases}
\]
Now consider a vector $y \in \R^{\mathcal{P}(X)}$ which defines a
positive semidefinite combinatorial moment matrix $M_n(y)$. Expand
  \[
    y = \sum_{J \in \mathcal{P}(X)} \alpha_J \chi^J_n,
  \]
 and so
    \[
      0 \leq (\psi^J_n)^{\sf T} M_n(y) \psi^J_n = \alpha_J. \qedhere
 \]
\end{proof}

Consider the $j$-th row of the system $Ax \geq b$ which is the
constraint $a_j^{\sf T} x \geq b_j$, where the $j$-th row $a_j$ of
matrix $A$ is viewed as a column vector. Using this information we can
define the $j$-th \emph{(truncated) localizing matrix} $M^j_{t}(y)$ of
$y \in \R^{\mathcal{P}_{2t+2}(X)}$ by
\[
  [M^j_{t}(y)]_{J,J'} = \sum_{i \in X} a_{ji} y(J \cup J' \cup \{i\}) - b_j y(J \cup J'),
\]
where $J, J' \in \mathcal{P}_t(X)$\footnote{In fact, the entries of
  $y$ which are indexed by $2t+2$-element subsets of $X$ are not used
  here. We keep them because the notation for the Lasserre hierarchy
  becomes simpler.}.

\subsection{The Lasserre hierarchy for $\ILP$}

For $t = 1, \ldots, n$, the \emph{$t$-th step of the Lasserre
  hierarchy} for $\ILP$ is the (value of the) following semidefinite
program
\[
  \begin{split}
  \las_t = \min\Big\{\sum_{i \in X} c_i y(\{i\}) \; : \; & y \in
  \R^{\mathcal{P}_{2t}(X)}, \; y(\emptyset) = 1, \\[-3ex]
  & M_t(y) \succeq 0, \;   M^j_{t-1}(y) \succeq 0 \; (j \in [m])\Big\}.
 \end{split}
\]

The Lasserre hierarchy provides a monotonically increasing sequence of
lower bounds for $\ILP$:
\[
  \ILP \geq \las_{n+1} \geq \las_{n} \geq \ldots \geq \las_1.
\]

The sequence is monotonically increasing because every feasible
solution $y \in \R^{\mathcal{P}_{2t}(X)}$ of $\las_t$ can be truncated
to a feasible solution $y' \in \R^{\mathcal{P}_{2t-2}(X)}$ of
$\las_{t-1}$ because $M_{t-1}(y')$ and $M^j_{t-2}(y')$ are
principal submatrices of $M_t(y)$ and $M^j_{t-1}(y)$.
 
We have $\ILP \geq \las_{n+1}$ because if $x^* \in \{0,1\}^n$ is an
optimal solution of $\ILP$, then $y = \chi^J_n$ with
$J = \{i \in X : x^*_i = 1\}$ is feasible for $\las_{n+1}$; Positive
semidefiniteness of $M_n(y)$ follows from the easy direction of Lemma~\ref{lem:crucial-lemma} and positive semidefiniteness of
$M_n^j(y)$ is again due to the homomorphism property of the
characteristic vector
\[
\begin{split}
[M^j_{n}(y)]_{J',J''} & \; = \; \sum_{i \in X} a_{ji} \chi^J_n(J' \cup J''
\cup \{i\}) - b_j \chi^J_n(J' \cup J'')\\
& \; = \;
\chi^J_n(J' \cup J'') \left(\sum_{i \in X} a_{ji} \chi^J_n(\{i\}) - b_j\right)\\
& \;  = \; \chi^J_n(J' \cup J'') \left(\sum_{i \in X} a_{ji} x^*_i - b_j\right),
\end{split}
\]
and hence
\[
  M^j_{n}(y) = \left(a_j^{\sf T} x^* - b_j\right) \chi^J_n
  (\chi^J_n)^{\sf T} \succeq 0,
\]
because $a_j^{\sf T} x^* \geq b_j$.

 Moreover, the Lasserre hierarchy converges to $\ILP$ in (at most)
 $n+1$ steps because the reverse inequality $\ILP \leq \las_{n+1}$
 holds. This can be seen by again applying Lemma~\ref{lem:crucial-lemma},
 which says that we can expand any feasible solution $y$ of $\las_{n+1}$ as
 $y = \sum_{J \in \mathcal{P}(X)} \alpha_J \chi^J_n$ with nonnegative
 coefficients $\alpha_J$. If $\alpha_J > 0$, then $x^J \in \{0,1\}^n$
 with $x^J_i = \chi^J_n(\{i\})$ is a feasible solution of $\ILP$, so that
 $c^{\sf T} x^J \geq \ILP$. This
 follows from the positive semidefiniteness of the localizing
 matrix. In particular, using the dual basis $\psi^J_n$,
 \[
   0 \leq (\psi^{J}_n)^{\sf T} M^j_n(y) \psi^{J}_n
= \alpha_{J} \left(\sum_{i \in X} a_{ji} \chi^J_n(\{i\}) - b_j\right)
 \]
 holds, which implies $a_j^{\sf T} x^J \geq b_j$.  The normalization
 $y(\{\emptyset\})=1$ translates into
 $\sum_{J \in \mathcal{P}(X)} \alpha_J = 1$ and therefore the
 objective value of $\las_{n+1}$ satisfies
 \[
   \sum_{i \in X} c_i y(\{i\}) =   \sum_{i \in X} c_i \sum_{J \in
     \mathcal{P}(X)} \alpha_J \chi^J_n(\{i\}) =
   \sum_{J \in \mathcal{P}(X)} \alpha_J c^{\sf T} x^J
   \geq \ILP.
 \]

\section{SDP hierarchy for geometric packing problems}
\label{sec:packing-hierarchy}

Let $X$ be a compact metric space with distance function $d$. By
$B(x,r) = \{x' \in X : d(x,x') < r\}$ we denote open metric balls with
radius $r$ and center $x$.

One important geometric parameter of $X$ and a positive $r$ is the
\emph{packing number}
\[
 \alpha(X,r) = \max\{|Z| : Z \subseteq X : B(x,r) \cap B(x',r) =
 \emptyset \text{
   for all } x, x' \in Z, x \neq x'\}.
\]

The idea behind the paper \cite{de-Laat-Vallentin-2015} of de Laat and
Vallentin is to consider $\alpha(X,r)$ as an infinite integer linear
program where every point in $X$ corresponds to a $0/1$-variable.
Then an appropriate semidefinite programming hierarchy is constructed
which can be used to determine upper bounds for $\alpha(X,r)$ and
which eventually converges to $\alpha(X,r)$. In this section we recall
the main steps of \cite{de-Laat-Vallentin-2015} because our
semidefinite hierarchy for covering builds on this packing
hierarchy. Whereas the packing hierarchy only uses moment constraints,
the covering hierarchy needs to incorporate localizing constraints.

The packing number can also be seen as the independence number of the,
potentially infinite, packing graph $P(X,r)$ with vertex set $X$ in which two
distinct vertices $x$, $x'$ are adjacent whenever
$B(x,r) \cap B(x',r) \neq \emptyset$. Generally, the \emph{independence
  number} of a graph is the maximum cardinality of an independent set,
where a subset of the vertex set is called an \emph{independent set}
if it does not contain a pair of adjacent vertices.

For a finite graph $G = (V, E)$, with vertex set
$V = \{1, \ldots, n\}$ and edge set $E$, one can formulate the
independence number of $G$ as an integer linear program with
$0/1$-variables, which is
\[
\alpha(G) =  \max\left\{\sum_{i=1}^n x_i : x \in \{0,1\}^n, x_i + x_j \leq 1 \text{
    for all } \{i,j\} \in E\right\},
\]
and apply the Lasserre hierarchy.

Laurent~\cite[Lemma 20]{Laurent-2003} observed that one can write the
Lasserre hierarchy for $\alpha(G)$ more compactly, only by considering combinatorial moment matrices, without using localizing matrices. In fact, this more compact relaxation is based on replacing the linear inequality $x_i + x_j \leq 1$ by the quadratic equality $x_i x_j = 0$ for each edge 
$\{i,j\} \in E$. This gives $y_{ij} = 0$ for $\{i,j\} \in E$ and in turn $y_I$ = 0 if $I$ contains an edge. 

To state this compact version of the Lasserre hierarchy, we
\textit{slightly change the definitions} introduced in
Section~\ref{sec:Lasserre-hierarchy}. We hope that this will not cause
confusion. The first change is that instead of working with the
complete power set of the vertex set we consider only independent
subsets. By $I_t$ we denote the set of all independent sets with at
most $t$ elements, then the $t$-th step of the Lasserre hierarchy is
\begin{equation}
  \label{eq:Lasserre-hierarchy-alpha}
\las^\alpha_t(G) = \max\left\{ \sum_{i=1}^n y(\{i\}) : y \in \R^{I_{2t}}, y\geq 0,
  \; y(\emptyset) =
      1, \; M_t(y) \succeq 0\right\},
  \end{equation}
  where we define $M_t(y) \in \R^{I_t \times I_t}$ by
  \[
    [M_t(y)]_{J,J'} = \begin{cases}
      y_{J \cup J'} & \text{if $J \cup J' \in I_{2t}$,}\\
      0 & \text{otherwise},
      \end{cases}
    \]
    which is the second change.  The third and last change, requiring
    $y$ to be a nonnegative vector, enables us to say that the first
    step $\las^\alpha_1(G)$ coincides with $\vartheta'(G)$, the
    strengthened version of Lov\'asz $\vartheta$-number
    \cite{Lovasz-1979} due to Schrijver~\cite{Schrijver-1979}. We obtain
    \[
      \vartheta'(G) = \las^\alpha_1(G) \geq \las^\alpha_2(G) \geq \ldots \geq
      \las^\alpha_{\alpha(G)}(G) = \alpha(G).
    \]

    In \cite{de-Laat-Vallentin-2015} this semidefinite programming
    hierarchy~\eqref{eq:Lasserre-hierarchy-alpha} is generalized to
    give a hierarchy for the packing number $\alpha(X,r)$. In
    particular, independence numbers of topological packing graphs are
    introduced, which include the packings graphs $P(X,r)$ on compact
    metric spaces.

    Since $X$ can have infinitely many points, some topology is needed
    to generalize the objects
    in~\eqref{eq:Lasserre-hierarchy-alpha}. The natural distance on
    nonempty subsets of $X$ is the \emph{Hausdorff distance}
\[
  d_H(J,J') = \inf\{\varepsilon : J \subseteq J'_{\varepsilon} \text{ and }
  J' \subseteq J_{\varepsilon}\},
\]
where $J_\varepsilon = \bigcup_{x \in J} B(x,\varepsilon)$ is the
$\varepsilon$-thickening of the set $J$. 

Using the Hausdorff distance on $I_t$ instead of $\mathcal{P}_t(X)$
has an important advantage. By considering $\alpha(X,r)$ we are
interested in cardinalities and generally elements in
$\mathcal{P}_t(X)$ with non-equal cardinality can be topologically
close. For instance, a sequence of two or more points in $\mathcal{P}_t(X)$, which are getting arbitarily close, converges to one limit point. However, because we work with $I_t$, and not with
$\mathcal{P}_t(X)$, independent sets having different cardinality, lie
in separate connected components; see the discussion in \cite[Section
2]{de-Laat-Vallentin-2015}.

In the generalization, the vector $y \in \R_{\geq 0}^{I_{2t}}$ is replaced by
a measure $\lambda \in \mathcal{M}(I_{2t})_{\geq 0}$ in the cone of non-negative Radon measures. The idea is that an independent set $S$ determines a feasible solution
\[
  \lambda = \sum_{R \in I_{2t} : R \subseteq S} \delta_R, \quad \text{where $\delta_R$ is the delta
measure at $R$,}
\]
of the $t$-th step of the packing hierarchy.

The map which assigns $y$ to the truncated combinatorial moment matrix
$M_t(y)$ is replaced by a dual construction using the
adjoint map. We take this detour because we do not know a more direct way.

By the Riesz representation theorem, the dual space of the space of
real-valued continuous functions $\mathcal{C}(I_{2t})$, equipped with
the supremum norm, is the space of signed Radon measures
$\mathcal{M}(I_{2t})$. In particular, the pairing 
\[
\langle \cdot , \cdot \rangle: \mathcal{C}(X) \times \mathcal{M}(X) \rightarrow \R, \langle f,\mu\rangle \coloneqq \int_X f d\mu
\]
is well-defined for every compact metric space $X$ such as $I_{2t}$. For the ease of notation, we will denote any such map by $\langle\cdot , \cdot \rangle$ and omit to specify $X$ in the remainder of this article.

Furthermore, we denote by $\mathcal{M}(I_{2t})_{\geq 0}$ the cone of nonnegative Radon measures and by $\mathcal{C}(I_{2t})_{\geq 0}$ the cone of nonnegative continuous functions. Note that these cones are their respective duals. Similarly, the
dual space of symmetric, real-valued continuous kernels
$\mathcal{C}(I_t \times I_t)_{\sym}$ is the space of symmetric Radon
measures $\mathcal{M}(I_t \times I_t)_{\sym}$. The dual cone of
positive definite kernels, that is kernels
$K \in \mathcal{C}(I_t \times I_t)_{\sym}$ which satisfy
\[
  (K(J_i, J_j))_{i,j=1}^m \succeq 0 \quad \text{ for all } \; J_1, \ldots, J_m
  \in I_{2t}, \; m \in \mathbb{N},
\]
we denote by $\mathcal{M}(I_t \times I_t)_{\succeq 0}$.

We define the map
\[
  A_t : \mathcal{C}(I_t \times I_t)_{\sym} \to \mathcal{C}(I_{2t})
  \quad \text{by} \quad
  A_t K(S) = \sum_{J,J' \in I_t : J \cup J' = S} K(J,J'),
\]
where $K \in \mathcal{C}(I_t \times I_t)_{\sym}$ and $S \in I_{2t}$.
A proof of the fact that $A_t$ is well-defined, that is $A_tK \in\mathcal{C}(I_{2t})$, can be found in \cite[Section 3.1]{Bekker-Oliveria-2023}.
Note that the above sum has at most $2^{2t}$ summands and consequently $\|A_tK\|_\infty \leq 2^{2t} \|K\|_\infty$ implying that $A_t$ is bounded and hence continuous.

Therefore, the adjoint map
\[
  A_t^* : \mathcal{M}(I_{2t}) \to \mathcal{M}(I_t \times I_t)_{\sym},
\]
(well-)defined by
\[
  \langle K, A^* \lambda \rangle = \langle A_t K, \lambda \rangle,
\]
is the natural replacement for $y \mapsto M_t(y)$.

After these preparations one can define the $t$-th step of the
\emph{packing hierarchy} by
\[
\las^\alpha_t(P(X,r)) = \max\{\lambda(I_{=1}) : \lambda \in \mathcal{M}(I_{2t})_{\geq 0}, \;
\lambda(\{\emptyset\}) = 1, \;
A_t^*(\lambda) \in \mathcal{M}(I_t \times I_t)_{\succeq 0} \},
\]
where $I_{=1}$ denotes the set of independent subsets with cardinality
exactly $1$.  Again, and this is the main result of
\cite{de-Laat-Vallentin-2015},
\[
        \vartheta'(P(X,r)) = \las^\alpha_1(P(X,r)) \geq \las^\alpha_2(P(X,r)) \geq \ldots \geq
      \las^\alpha_{\alpha(X,r)}(P(X,r)) = \alpha(X,r)
\]
holds.

The first step coincides with the generalization of the
$\vartheta'$-number for the graph $P(X,r)$, defined in
\cite{Bachoc-Nebe-Oliveira-Vallentin-2009}.  If for instance the
compact metric space is the unit sphere, $X = S^{n-1}$, then the first
step coincides with the linear programming bound of Delsarte,
Goethals, Seidel \cite{Delsarte-Goethals-Seidel-1977}.

The convergence follows from a Choquet-type variant (cf.\ Simon
\cite{Simon-2011}) of the crucial
Lemma~\ref{lem:crucial-lemma}, where finite convex combinations are
replaced by integrals with respect to probability measures, which
reads:

\begin{lemma} (de Laat, Vallentin \cite[Proposition
  1]{de-Laat-Vallentin-2015})
\label{lem:crucial-lemma-generalized}
  Let $I = I_{\alpha(X,r)}$ be the set of all independent sets of
$P(X,r)$.  Let $\lambda \in \mathcal{M}(I)$ be so that
$\lambda(\{\emptyset\}) = 1$ and
$A_{\alpha(X,r)}^* \lambda \in \mathcal{M}(I \times I)_{\succeq
  0}$. Then there exists a unique probability measure
$\sigma \in \mathcal{M}(I)_{\geq 0}$ so that $\lambda$ can be represented
as
  \[
    \lambda = \int \chi^R \, d\sigma(R), \quad \text{where} \quad \chi^R = \sum_{Q \subseteq R} \delta_Q.
  \]
\end{lemma}
 
 Next to putting the linear programming bound into perspective and
 giving theoretical convergence, the packing hierarchy has turned out to be
 useful in various applications, especially for the unit sphere or for
 Euclidean space where the presence of symmetries makes it possible to
 simplify the computations.  Bachoc, Vallentin
 \cite{Bachoc-Vallentin-2008} used semidefinite constraints from
 $\las^\alpha_2$ to find new upper bounds for the kissing number. De Laat
 \cite{de-Laat-2020} used the second step of the hierarchy in the
 context of energy minimization on the sphere $S^2$. De Laat, Leijenhorst, de Munick Keizer \cite{de-Laat-Leijenhorst-de-Munick-Keizer-2024} used the second step
 to prove the uniqueness of the four-dimensional kissing configuration. De Laat,
 Machado, de Munick Keizer
 \cite{de-Laat-Machado-de-Munick-Keizer-2022} computed $\las^\alpha_2$ and
 $\las^\alpha_3$ for the problem of equiangular lines; they were even able to
 turn these computations in an asymptotic analysis.  Cohn, Salmon
 \cite{Cohn-Salmon-2021} showed that the packing hierarchy
 converges to the packing density for Euclidean space. Cohn, de Laat,
 Salmon \cite{Cohn-de-Laat-Salmon-2022} used constraints from $\las^\alpha_2$
 to find new upper bounds for sphere packing densities in dimensions $4$
 through $7$ and $9$ through $16$.

 \section{SDP hierarchy for geometric covering problems}
\label{sec:covering-hierarchy}

The finite-graph analog of the covering number is the domination
number. For a finite graph $G = (V, E)$, with vertex set
$V = \{1, \ldots, n\}$ and edge set $E$ a subset of the vertex set
$D \subseteq V$ is called a \emph{dominating set} if for every vertex
$i \in V$ either the vertex lies itself in the dominating set or there is vertex
$j \in D$, in the dominating set, which is adjacent to $i$. The
\emph{domination number} $\gamma(G)$ of $G$ is the cardinality of a
smallest dominating set. Let us further define the \emph{extended neighbourhood} of a vertex $j$ by $N(j)\coloneqq \{i\in V: i=j \lor \{i,j\}\in E\}$. Then, one can formulate $\gamma(G)$ as an integer
linear program with $0/1$-variables:
\[
 \gamma(G) =  \min\left\{\sum_{i=1}^n x_i : x \in \{0,1\}^n,
    \sum_{i \in N(j)} x_i \geq 1 \; (j \in V) \right\},
\]
and apply the Lasserre hierarchy. Here, the $t$-th step of the Lasserre
hierarchy is
\[
  \begin{split}
  \las^\gamma_t(G) = \min\Big\{ \sum_{i=1}^n  y(\{i\}) \; : \; & y \in
  \R^{\mathcal{P}_{2t}(V)}, y(\emptyset) = 1,\\[-3ex]
  & M_t(y) \succeq 0, \; M^j_{t-1}(y) \succeq 0 \; (j \in V)\Big\}
  \end{split}
\]
where the linear constraint that every vertex $j \in V$ either should
lie in the dominating set or should be adjacent to a vertex in the
dominating set determines a truncated localizing matrix
\begin{equation}
[M^j_{t-1}(y)]_{J,J'} = \sum_{i \in N(j)} y(J \cup J' \cup
\{i\}) - y(J \cup J'),\label{eq:truncated_localizing_matrix_finite}
\end{equation}
where $J, J' \in \mathcal{P}_{t-1}(V)$.

In the geometric setting we consider the \emph{covering graph} $C(X,r)$
with vertex set $X$ in which two vertices $x, x' \in X$ are adjacent
whenever $d(x,x') \leq r$. Then, as dominating sets in
$C(X,r)$ determine the centers of balls with radius $r$ of
a covering of $X$, and vice versa, we have
$\gamma(C(X,r)) = \mathcal{N}(X,r)$.

\smallskip

We want to define a semidefinite hierarchy which converges in finitely
many steps to the covering number $\mathcal{N}(X,r)$. For the proof we
will make use of Lemma~\ref{lem:crucial-lemma-generalized}. Therefore,
we must assume \emph{a priori} that we know a lower bound $\varepsilon > 0$ such that an
optimal covering $X$ determines a packing of balls with radius
$\varepsilon$. Clearly, such a positive lower bound always exists as
optimal coverings consist of finitely many pairwise distinct balls. We further comment on $\varepsilon$ at the end of this section.

We define $I_t$ to be the set of subsets of $X$ with at
most $t$ elements which are independent in the packing graph $P(X,\varepsilon)$.
To emphasize that in the following everything depends on $\varepsilon$ it would make sense to change the notation from $I_t$ to $I^{\varepsilon}_t$. However, we have chosen not to do so in order to avoid cluttering the notation.

\smallskip

For $z \in X$ we define the map
\[
  B^z_t : \mathcal{C}(I_t \times I_t)_{\sym} \to \mathcal{C}(I_{2t+2})
\]
by
\[
  B^z_t K(S) = \sum_{x \in \overline{B}(z,r), J, J'\in I_t : J \cup J' \cup
    \{x\} = S} K(J,J')
  - \sum_{J,J' \in I_t : J \cup J' = S} K(J,J'),
\]
where $K \in \mathcal{C}(I_t \times I_t)_{\sym}$ and $S \in I_{2t+2}$. Similarly as operator $A_t$ we observe that 
$\|B_t K \|_\infty \leq ((2t+2) 2^{2t} + 2^{2t}) \|K\|_\infty$ which implies $B_t^z$ being bounded and continuous. Hence, the adjoint map
$(B^z_t)^* : \mathcal{M}(I_{2t+2}) \to \mathcal{M}(I_t \times
I_t)_{\sym}$ gives the analog of the map \eqref{eq:truncated_localizing_matrix_finite} which assigns $y$ to the truncated localizing matrix.

\begin{definition}
  For $t = 1, 2, \ldots, $ we define the $t$-th step of the \emph{covering
  hierarchy} by
\[
\begin{split}
\las^\gamma_t(C(X,r)) = \inf\Big\{\lambda(I_{=1})   \; : \; & \lambda \in \mathcal{M}(I_{2t})_{\geq 0}, \;
\lambda(\{\emptyset\}) = 1,\\
& A_t^*(\lambda) \in \mathcal{M}(I_t \times I_t)_{\succeq 0}, \\
& (B^z_{t-1})^*(\lambda) \in  \mathcal{M}(I_{t-1} \times
I_{t-1})_{\succeq 0} \; (z \in X) \Big\}.
\end{split}
\]
\end{definition}

\begin{theorem}
 \label{thm:main}
The covering hierarchy gives a monotonically increasing sequence of
lower bounds for the covering number and converges to the covering
number in at most $\alpha(P(X,\varepsilon))$ steps:
\[
\las^\gamma_1(C(X,r)) \leq \las^\gamma_2(C(X,r)) \leq \ldots
\leq \las^\gamma_{\alpha(P(X,\varepsilon))}(C(X,r)) =
\mathcal{N}(X,r).
\]
\end{theorem}

\medskip

Now we return to the proof of the main theorem, which we split into three lemmas.

\medskip

Every step of the covering hierarchy gives a lower bound for the
covering number.

\begin{lemma}
  For every $t = 1, 2, \ldots$ the inequality
  $\las^\gamma_t(C(X,r)) \leq \mathcal{N}(X,r)$ holds.
\end{lemma}

\begin{proof}
  We verify this inequality by showing that every finite covering $Y$
  of $X$ with minimum separation $\varepsilon$ gives a feasible solution
\[
\lambda = \sum_{Z\in I_{2t}:\ Z\subseteq Y}\delta_Z
\]
of the $t$-step of the covering hierarchy with objective
$\lambda(I_{=1}) = |Y|$.

Indeed, we have
\[
  \lambda\left(\{\emptyset\}\right) =\sum_{Z\in I_{2t}: Z\subseteq
    Y}\delta_Z
  \left(\{\emptyset\}\right)=\delta_{\emptyset}\left(\{\emptyset\}\right)=1.
\]

Then, for every $K \in \mathcal{C}(I_t\times I_t)_{\succeq 0}$:
\[
\begin{split}
		\langle A_t^* \lambda ,K \rangle & = \langle \lambda,A_t K\rangle\\
		& = \sum_{Z\in I_{2t} :Z\subseteq Y} \sum_{J,J'\in I_t
                  : Z = J\cup J'} K(J,J')\\
		& = \sum_{J,J'\in I_t : J,J'\subseteq Y} K(J,J')\\
                & \geq 0,
\end{split}
\]
since $K\succeq 0$.

Moreover, for every $z \in X$ and every $K' \in \mathcal{C}(I_{t-1}\times I_{t-1})_{\succeq 0}$
\[
\begin{split}
  & \langle (B_{t-1}^z)^* \lambda, K' \rangle\\
  = \; & \langle
                \lambda, B_{t-1}^z K'\rangle \\
		= \; & \sum_{Z\in I_{2t} :  Z\subseteq Y} \Big(\sum_{x \in
                  \overline{B}(z,r),\ J,J'\in I_{t-1} :  Z = J \cup J'
                  \cup \{x\} } K'(J,J')  -\sum_{J,J'\in I_{t-1} : Z = J \cup J'} K'(J,J')\Big)\\
          =\; & \sum_{x\in \overline{B}(z,r)} \sum_{J,J'\in I_{t-1}
          	:\ \substack{J\cup J' \cup \{x\}\in I_{2t}\\ J\cup J'\cup \{x\} \subseteq Y}} K'(J,J') -\sum_{J,J'\in I_{t-1}:\ \substack{J\cup J' \in I_{2t}\\ J,J'\subseteq Y}} K'(J,J')\\
          \overset{*}{=} \; & \sum_{x\in Y \cap \overline{B}(z,r)} \sum_{J,J'\in I_{t-1}
          	:\ J,J'\subseteq Y} K'(J,J') -\sum_{J,J'\in I_{t-1} :
          	J,J'\subseteq Y} K'(J,J')\\
                 = \; & (|Y \cap \overline{B}(z,r)| - 1) \sum_{J,J'\in I_{t-1}:
                  J,J'\subseteq Y} K'(J,J')\\
          \geq \; & 0,
\end{split}
\]
where the equation marked by * exploits the fact that $A\subseteq Y$ implies that $A$ has minimum separation $\varepsilon$. Moreover, the last inequality is implied by the facts that
$Y \cap \overline{B}(z,r) \neq \emptyset$ since $Y$ is a covering and $K\succeq 0$.
\end{proof}

The covering hierarchy is monotonically increasing with $t$.

\begin{lemma}
For every $t = 1, 2, \ldots$ the inequality
  $\las^\gamma_t(C(X,r)) \leq \las^\gamma_{t+1}(C(X,r))$ holds.
\end{lemma}

\begin{proof}
  Increasing $t$ strengthens the bound by imposing more constraints to
  the set of feasible solutions. To be precise, if the measure
  $\lambda \in \mathcal{M}(I_{2t+2})_{\geq 0}$ is feasible for
  $\las^\gamma_{t+1}(C(X,r))$, then its restriction to
  $\mathcal{M}(I_{2t})$ is also feasible for
  $\las^\gamma_t(C(X,r))$. Furthermore, $\lambda$ and its
  restriction have the same objective value, namely $\lambda(I_{=1})$.
\end{proof}

The covering hierarchy converges to the covering number.

\begin{lemma}
  Equality
  $\las^\gamma_{\alpha(P(X,\varepsilon))}(C(X,r)) = \mathcal{N}(X,r)$
  holds.
\end{lemma}

\begin{proof}
  Set $\alpha = \alpha(P(X,\varepsilon))$. We consider a feasible measure
  $\lambda\in \mathcal{M}(I_{\alpha})_{\geq 0}$ for the program $\las^\gamma_{\alpha}(C(X,r)) $ and its
  representation $\lambda=\int \chi^R \, d\sigma(R)$, where we recall $\sigma$ to be the unique probability measure which exists due to Lemma~\ref{lem:crucial-lemma-generalized}.

  Now we have to verify that if a subset
$R \in I_{\alpha}$ belongs to the support of $\sigma$, then $R$ determines a covering.   
(As usual, the \emph{support} of measure $\sigma$ is the smallest closed subset of $I_{\alpha}$
such that $\sigma$ assigns measure zero to its complement.)
  
  For a nonnegative continuous function $g \in \mathcal{C}(I_\alpha)_{\geq 0}$ define $f \in \mathcal{C}(I_\alpha)$ by
  \[
  f(Q) = \sum_{P \subseteq Q} (-1)^{|Q \setminus P|} \sqrt{g(P)},
  \]
  so that (see \cite{de-Laat-Vallentin-2015} for the computation involving the inclusion-exclusion principle and see \cite{Bekker-Oliveria-2023} for the argument that $f$ is continuous)
  \[
  \sum_{Q \subseteq R} f(Q) = \sqrt{g(R)}.
  \]
 
 Then, $f \otimes f \in
\mathcal{C}(I_\alpha \times I_\alpha)_{\succeq 0}$ 
is a positive definite kernel.  Then, by the feasibility of $\lambda$ we have $\langle (B^z_{\alpha})^* \lambda, f \otimes f \rangle \geq 0$ for every $z \in X$. On the other hand,
\[
  \begin{split}
& \langle (B^z_{\alpha})^* \lambda, f \otimes f\rangle \\
    \; = \; &   \int \langle \chi^R, B^z_{\alpha}(f \otimes f)\rangle \, d\sigma(R)\\
 \; = \;  &  \int  \sum_{x \in \overline{B}(z,r), J, J'\in I_\alpha : J \cup J' \cup
    \{x\} \subseteq R} (f \otimes f)(J,J')   - \sum_{J,J' \in I_\alpha : J
    \cup J' \subseteq R} (f \otimes f) (J,J')\, d\sigma(R)\\
  \; = \; & \int |\overline{B}(z,r) \cap R|  \sum_{ J, J'\in I_\alpha : J \cup J' 
    \subseteq R} (f \otimes f)(J,J')   - \sum_{J,J' \in I_\alpha : J
    \cup J' \subseteq R} (f \otimes f)(J,J') \, d\sigma(R)\\
  \; = \; & \int (|\overline{B}(z,r) \cap R| - 1) \left(\sum_{Q \subseteq R} f(Q) \right)^2 \, d\sigma(R)\\
  \; = \; & \int (|\overline{B}(z,r) \cap R| - 1) g(R) \, d\sigma(R)\\
  \end{split}
\]
If $R$ does not determine a covering, then there exists $z \in X$ so that
  $\overline{B}(z,r) \cap R = \emptyset$. In this case we have
  $|\overline{B}(z,r) \cap R| - 1 = -1$ and hence the above integral can be made negative by choosing an appropriate function $g$. 
  (For this note if $\overline{B}(z, r) \cap R = \emptyset$,
  then $\overline{B}(z, r) \cap R' = \emptyset$ for all $R'$ in some neighborhood of $R$.)
  This contradicts the feasibility of $\lambda$. We conclude that every $R$ in the support of $\sigma$ determines a
  covering.

We finish the proof of the lemma by looking at the objective value of
$\lambda$ which is
\[
  \lambda(I_{=1}) =\int \chi^{R}(I_{=1}) \, d\sigma(R) = \int |R| \, d\sigma(R) \geq \int \mathcal{N}(X,r)
  \, d\sigma(R) =\mathcal{N}(X,r),
\]
because every $R$ in the support of $\sigma$ determines a covering and
$\sigma$ is a probability measure.
\end{proof}

Finally, we would like to comment on the
necessary and partially implicit $\varepsilon$. The above proof showed that the hierarchy converges for any $\varepsilon$ chosen sufficiently small. However, when applying the hierarchy in practice, a specific choice of $\alpha$ is required. Determining an appropriate $\alpha$ involves  additional geometric considerations which usually are  challenging  to establish. In fact, initial numerical experiments with this hierarchy, conducted by Nando Leijenhorst, indicate that the hierarchy is quite sensitive to the choice of $\varepsilon$.

\section{The dual covering hierarchy}
\label{sec:dual-covering-hierarchy}

In this section we assume that the automorphism group $G$ of $X$, consisting of all bijective maps from $X$ to $X$ that preserve the distance,
acts transitively on $X$. Thus,  $X$ is the homogenous space $G/H$ where
$H$ is the stabilizer subgroup of a point in $X$. This is a natural
assumption which is fulfilled in many examples; for example when $X$
is the unit sphere $S^{n-1}$ and $G$ is the orthogonal group
$\mathrm{O}(n)$ and $H$ is  isomorphic to $\mathrm{O}(n-1)$. Then the
Haar measure of $G$ induces a probability measure $\omega$ on $G/H$
which satisfies conditions (a) and (b) given in Section~\ref{sec:introduction}.

Looking at the covering hierarchy this transitivity has the advantage, due to
convexity, that the constraints coming from the localizing matrix only
have to be required for one arbitrary point of $X$. Let $e \in X$ be this
point, then
\[
  \begin{split}
  \las^\gamma_t(C(X,r)) = \inf\Big\{\lambda(I_{=1})   \; : \; & \lambda \in \mathcal{M}(I_{2t})_{\geq 0}, \;
\lambda(\{\emptyset\}) = 1,\\
& A_t^*(\lambda) \in \mathcal{M}(I_t \times I_t)_{\succeq 0}, \\
& (B^e_{t-1})^*(\lambda) \in  \mathcal{M}(I_{t-1} \times
I_{t-1})_{\succeq 0} \Big\}.
\end{split}
\]

We define the $t$-th step of the dual covering hierarchy by
\[
\begin{split}\label{dual_hierarchy}
\las^{\gamma,*}_t(C(X,r)) = \sup\Big\{ \eta \; : \; & \eta \in \R,\\[-1ex]
& \mathbf{1}_{I_{=1}}(S)- \eta\mathbf{1}_{\{\emptyset\}}(S)\notag  -A_t
K(S) - B_{t-1}^e K'(S) \geq 0\\
& \qquad \text{ for all } S\in I_{2t}, \\[-1ex]
&  K\in \Ccal (I_t \times I_t)_{\succeq 0},\ K'\in \mathcal{C} (I_{t-1} \times I_{t-1})_{\succeq 0} \Big\}.
\end{split}
\]

One can easily check that weak duality $\las^{\gamma,*}_t \leq \las^{\gamma}_t$
holds.  Although we do not need this here, we note that also strong
duality holds; this was shown by Rolfes in \cite{Rolfes-2019} using
the framework of infinite dimensional conic optimization presented in
Barvinok \cite[Chapter IV]{Barvinok-2002}.

With the help of the dual covering hierarchy $\las_t^{\gamma,*}$ it is
further possible to show that the first step of the hierarchy
coincides with the volume bound.

\begin{theorem}
\label{thm:first-step-volume-bound}
  If the automorphism group $G$ acts transitively on $X$, then the first
  step of the covering hierarchy collapses to the volume bound,
  \[
    \las^{\gamma}_1(C(X,r))=\las^{\gamma,*}_1(C(X,r))=\frac{1}{\omega_r}.
  \]
\end{theorem}

\begin{proof}
  We define the measure $\lambda\in \mathcal{M}(I_2)_{\geq 0}$ as follows.
  We take two independent sets $R, R'$; one with $|R| < 1/\omega_r$ and one with $|R'| > 1/\omega_r$.
  Then we take the convex combination $\lambda$ of the characteristic measures $\chi^R =  \sum_{Q \subseteq R} \delta_Q$ and $\chi^{R'}$, restricted to  $\mathcal{M}(I_2)_{\geq 0}$,
  such that $\lambda(I_{=1}) = 1/\omega_r$. Then $\lambda$  satisfies $\lambda(\{\emptyset\}) = 1$ and $\lambda(I_{=1}) = 1/\omega_r$ and $A_1^* \lambda \in \mathcal{M}(I_1 \times I_1)_{\succeq 0}$.
  Now we take the average of $\lambda$ over the group $G$, so that the averaged $\lambda$ also satisfies the covering constraint 
  $(B_0^e)^* \lambda \in \mathcal{M}(I_0 \times I_0)_{\succeq 0}$, as every point of $X$ is covered on average by at least one ``point" in $\lambda$. Therefore, $\las^{\gamma}_1(C(X,r))\leq \frac{1}{\omega_r}$.

For the reverse inequality we have
\[
  B_0^e K'(\emptyset)=-K'(\emptyset,\emptyset),
\]
and
\begin{align*}
	B_0^e K'\left(\{x'\}\right)& = \sum_{x\in \overline{B}(e,r):\ \{x\}= \{x'\} } K'(\emptyset,\emptyset) \\
                                  & =
\begin{cases}
  K'(\emptyset,\emptyset) & \text{ if } x' \in \overline{B}(e,r),\\
  0 & \text{ otherwise,}
  \end{cases}
\end{align*}
and
\[
  B_0^e K'\left(\{x'_1,x'_2\}\right) =\sum_{x\in
    \overline{B}(e,r):\ \{x\}= \{x'_1,x'_2\} } K'(\emptyset,\emptyset)
= 0.
\]  
Thus,
\[
  K=0, \; K'(\emptyset,\emptyset)=\frac{1}{\omega_r} \; \text{ and } \;
  \eta=K'(\emptyset,\emptyset)
\]
is a feasible solution for $\las_1^{\gamma, *}(C(X,r))$. So
$\las^{\gamma,*}_1(C(X,r)) \geq \frac{1}{\omega_r}$ and by weak
duality $\las^{\gamma}_1(C(X,r)) \geq \frac{1}{\omega_r}$.
\end{proof}

For the packing number $\alpha(S^{n-1}, r)$ for the unit sphere, the
first step of the packing hierarchy $\las^{\alpha}_1(P(S^{n-1}, r))$
coincides with the linear programming bound $\vartheta'(P(S^{n-1},r))$
of Delsarte, Goethals, Seidel, which often provides strong and
sometimes even tight upper bounds. Now
Theorem~\ref{thm:first-step-volume-bound} gives an explanation why the
corresponding first step of the covering hierarchy
$\las^{\gamma}_1(C(S^{n-1}, r))$, which also is, after symmetry
reduction, a linear programming bound, is not a strong bound for the
covering number. On the other hand, the first step of the dual
covering hierarchy is implicitly used in the analysis of the greedy
approach to construct efficient coverings,
see~\cite{Rolfes-Vallentin-2018}.

\section*{Acknowledgements}

The authors are grateful to David de Laat, Nando Leijenhorst, and Fernando M\'ario de Oliveira Filho for helpful discussions and for suggesting corrections. They also wish to thank the two anonymous referees for their detailed and constructive feedback, which significantly improved the paper.

The first author was supported by the Troms{\o} Research Foundation project ``Symmetry in Real Algebraic Geometry'' and the UiT Aurora project MASCOT. The second author was supported by a "Short-Term Grant" funded by DAAD. The third named author is partially supported by the SFB/TRR 191
``Symplectic Structures in Geometry, Algebra and Dynamics'' funded by
the DFG.

\end{document}